\documentclass{amsart}
\usepackage[a4paper, left=3cm, right=3cm, top = 3cm, bottom = 3cm ]{geometry}
\usepackage{latexsym,array,delarray,amsthm,amssymb,psfrag}
\usepackage{tikz-cd}
\usepackage{amsmath,amscd}
\usepackage[all]{xy}
\usepackage{seqsplit}
\usepackage[colorlinks=false,linktocpage=true]{hyperref}
\usepackage[utf8]{inputenc}
\usepackage{graphicx}
\usetikzlibrary{matrix, calc, arrows}

\newtheorem{theorem}{Theorem}[section]
\newtheorem{lemma}[theorem]{Lemma}
\newtheorem{proposition}[theorem]{Proposition}
\newtheorem{corollary}[theorem]{Corollary}
\theoremstyle{definition}
\newtheorem{definition}[theorem]{Definition}

\theoremstyle{remark}
\newtheorem{remark}[theorem]{Remark}
\newtheorem*{theorem*}{Theorem}

\numberwithin{equation}{section}

\begin{document}

\title{Notes on equivariant higher Chow groups}

\author{Nguyen Manh Toan}
\curraddr{Institut für Mathematik, Universität Osnabrück, Albrechtstr. 28a, 
49076, Osnabrück.}
\email{toan.nguyen@uni-osnabrueck.de}




\begin{abstract}
In this short note, we prove a comparision theorem between Levine-Serp\'e's equivariant higher Chow groups of an algebraic variety equipped with an action of a finite group and ordinary higher Chow groups of its fixed points.
As a consequence, we show that the equivariant motivic spectral sequence degenerates rationally. This yields a Riemann-Roch Theorem for equivariant algebraic $K$-theory.
\end{abstract}

\maketitle\

\section{Introduction}
Let $k$ be a field and $G$ a finite group of order coprime to the exponential characteristic of $k$. Levine-Serp\'e have defined in \cite{LS1} equivariant higher Chow groups $CH_{p}(G,X,r)$ 
for any separated noetherian scheme $X$ which is essentially of finite type over $k$ equipped with an action of $G$ (see Definition \ref{higherChowgroup} below). 
They are a generalization of Bloch’s higher Chow groups in the equivariant setting. 
Equivariant higher Chow groups can be considered as a motivic Borel-Moore homology theory on the category of $G$-schemes over $k$. 
Moreover, these groups have a close relation with the higher $K$-theory of $G$-equivariant coherent sheaves on $X$ (cf. \cite{RT1}) by a spectral sequence
\begin{equation} \label{LevineSerpe}
E^{p,q}_1 = CH_{-p}(G,X,-p-q) \Rightarrow G_{-p-q}(G,X)
\end{equation}
(\cite[Corollary 3.8]{LS1}). In the case of the trivial group, this reduces to the motivic spectral sequence constructed by Bloch-Lichtenbaum \cite{BL1}, Friedlander-Suslin \cite{FS1} and Levine \cite{ML3}.

The groups $CH_p(G,X,r)$ are interesting objects which contain information about $X$ as well as the action of $G$ on $X$. They are, of course, very hard to compute in general.

In this note, we establish the following reconstruction theorem for equivariant higher Chow groups.
\begin{theorem*}[Theorem \ref{MainTheorem}]
Let $n$ be the order of $G$ and assume that $k$ contains $n$-th roots of unity. For any $G$-scheme $X$ over $k$, there is a natural isomorphim of $\mathbb{Z}[1/n]$-modules
\begin{equation} \label{1stIso}
\pi(X) : \prod_{\sigma \in \Gamma} (CH_p (X^{\sigma}, r) \otimes \tilde{R}\sigma)^{N_G(\sigma)} \xrightarrow{\sim} CH_p(G, X, r) \otimes \mathbb{Z}[1/n]
\end{equation}
where $\Gamma$ is a set of representatives for the conjugacy classes of cyclic subgroups of $G$,
$X^{\sigma}$ is the \textit{fixed point subscheme} of $X$ under the action of $\sigma$, 
$\tilde{R}\sigma$ is the $m$-th cyclotomic subring of the representation ring $R\sigma$ tensored with $\mathbb{Z}[1/n]$ where $m$ is the order of $\sigma$
and $N_G(\sigma)$ is the \textit{normalizer} of $\sigma$ in $G$.
\end{theorem*}

The left-hand side of \eqref{1stIso} can be easily expressed in the form of 
$$
(\prod_{g \in G} CH_p (X^{g}, r) \otimes \mathbb{Z}[1/n])^G.
$$
For $X$ smooth, $r=0$ with rational coefficients, $(\prod_{g \in G} CH_*(X^g,0) \otimes \mathbb{Q})^G$ is exactly the 
(\textit{small}) \textit{orbifold Chow ring} $CH_{*}^{orb}([X/G])$ of the global quotient $[X/G]$ studied by Abramovich–Graber–Vistoli \cite{AGV1} and Jarvis-Kaufmann-Kimura \cite{JKK1} which is an algebraic version of Chen–Ruan cohomology.

Theorem \ref{MainTheorem} is an analog of Vistoli’s theorem for equivariant algebraic $K$-theory \cite[Theorem 1 and 2]{AV1} and Segal’s theorem for equivariant topological $K$-theory \cite{HH1}. 
Using this result we show that the spectral sequence \eqref{LevineSerpe} degenerates with rational coefficients for $X$ smooth which yields
a Riemann-Roch theorem for equivariant algebraic $K$-theory.
\begin{theorem*}[Theorem \ref{BigConsequence}]
If $X$ is a smooth $G$-scheme over $k$, there is a natural isomorphism of $\mathbb{Q}$-vector spaces
\begin{equation}
\bigoplus_{p,q} CH_{p}(G,X,q) \otimes \mathbb{Q} \cong \bigoplus_{q}K_{q}(G,X) \otimes \mathbb{Q}.
\end{equation}
\end{theorem*}
Moreover, Theorem \ref{MainTheorem} provides two multiplicative structures on rational equivariant higher Chow groups. Hence we obtain complete answers to questions proposed in \cite[Introduction]{LS1}. 

This paper is organized as follows. In Section 2, we will briefly review Levine-Serp\'e's construction of equivariant higher Chow groups and discuss various functorial properties. 
Most of the things can be found in \cite{LS1} except the construction and functorialities of the \textit{induction map} that is new. 
In Section 3, we establish the comparision morphism from higher Chow groups of fixed point subschemes to equivariant higher Chow groups. This is modelled by Vistoli's construction for the $K$ groups \cite[Section 3]{AV1}.
Using localization property for (equivariant) higher Chow groups, we will show that the comparision morphism is an isomorphism.
In Section 4, we will lift Vistoli's construction to the $K$-theory spectra. This allows us to compare the equivariant motivic spectral sequence with the ordinary one. 
From there, some applications will be derived, including a Riemann-Roch theorem for equivariant algebraic $K$-theory and multiplicative structures on equivariant higher Chow groups with rational coefficients.
We will also see how to understand Levine-Serp\'e's equivariant higher Chow groups from orbifold theories.

This paper was written during my visit at the Mittag-Leffler Institute inside the Research program on Algebro-Geometric and Homotopical Methods.
I would like to thank the Institute and organizers, especially Paul Arne Østvær, for providing a stimulating, supportive environment and financial support. 
Discussions with Alexey Ananyevskiy, Federico Binda, Marc Levine and Mathias Wendt were very helpful. A brief visit to the Hausdorff Research Institute for Mathematics enabled me to profit from discussions with Lie Fu.
\subsection*{Notation and conventions}

Throughout this paper, $k$ is a fixed base field and $G$ is a finite group whose order is coprime to the exponential characteristic of $k$.

We write $\mathbf{Sch}^G_k$ for the category whose objects are seperated schemes essentially of finite type and quasi-projective over $k$ equipped with a left $G$-action and morphisms are $G$-equivariant morphisms.
Set $\Lambda: = \mathbb{Z}[1 /|G|]$ for our coefficient ring.

For any $X \in \mathbf{Sch}^G_k$, the \textit{abelian category} of $G$-equivariant coherent sheaves on $X$ is denoted by $\mathbf{Coh}_G(X)$ and 
the \textit{exact category} of $G$-equivariant vector bundles on $X$ is denoted by $\mathbf{Vect}_G(X)$. 
In the case of the trivial group $G$, we denote $\mathbf{Coh}(X)$ for $\mathbf{Coh}_G(X)$.

The connected $K$-theory spectra of $\mathbf{Coh}_G(X)$ and $\mathbf{Vect}_G(X)$ are denoted by $G(G,X)$ and $K(G,X)$, respectively.  
We will use the notation $G_i(G,X)$ to denote the $i$-th homotopy group of $G(G,X)$ and similarly for $K(G,X)$. 
When $X$ is smooth, the natural inclusion $\mathbf{Vect}_G(X) \hookrightarrow \mathbf{Coh}_G(X)$ induces an isomorphim $K_i(G,X) \xrightarrow{\sim} G_i(G,X)$. This identification will be used frequently in this paper. 

For any group $H$, the ring of representations of $H$ over $k$ tensored with $\Lambda$ is denoted by $RH$, i.e., $RH: = K_0(H, \mathrm{Spec}(k)) \otimes \Lambda$. 
\section{Equivariant higher Chow groups}

We recall Levine-Serp\'e's definition of equivariant higher Chow groups \cite{LS1} and discuss some of their properties which will be used in the following sections.
\subsection{Construction}

Let $\Delta^{\bullet}$ be the standard cosimplicial scheme with
$$
\Delta^{r}: = {\rm Spec} \big( k[t_0, \ldots , t_r] / (\sum t_i = 1) \big)
$$
equipped with the trivial action of $G$. A $face$ of $\Delta^r$ is a closed subscheme defined by $t_{i_1} = \ldots = t_{i_j} = 0$. 

For any $X \in \mathbf{Sch}^G_k$, we set 
$$
S_{(p)}^{G,X}(r): = \left\{ W \subset X \times \Delta^r \middle|
\begin{array}{c l}  W \mbox{ is a closed } G-\mbox{stable subset} \\
\mbox{and }\mbox{dim}_{X\times F} (W \cap X\times F) \le p + \mbox{dim} F \\ 
\mbox{for all faces }F \subset \Delta^r 
\end{array} \right\}.
$$

The group $G$ acts obviously on the set $(X \times \Delta^r)_{(p+r)}$ of dimension $(p+r)$ points on $X \times \Delta^{r}$. Let
$$
X^{G}_{(p)}(r): = \{[x] \in (X \times \Delta^r)_{(p+r)} /G \ | \ \overline{G . x} \in S^{G,X}_{(p)}(r)\}
$$
where $\overline{G .x}$ stands for the closure of the orbit $G.x$ in $X \times \Delta^r$.

We define 
$$
z_p (G,X,r): = \bigoplus_{[x] \in X^G_{(p)}(r)} K_0(G_x, \mathrm{Spec} (k(x))),
$$
where $k(x)$ is the residue field of $x$ and $G_{x}$ is the \textit{set-theoretic stabilizer group} of $x$ in $G$, i.e.,
$
G_x:= \{g \in G | \ gx = x \}.
$ 
The assigment $r \to z_p(G,X,r)$ forms a simplicial abelian group (\cite[Proposition 3.2]{LS1}) which is denoted by $z_p(G,X,\bullet)$. 

\begin{definition} \cite[Definition 3.4]{LS1} \label{higherChowgroup}
The \textit{equivariant cycle complex (of Bredon type)} $z_p(G,X,*)$ is the complex associated to $z_p(G,X,\bullet)$.
The \textit{equivariant higher Chow groups (of Bredon type)} are defined by
$$
CH_{p}(G,X,r): = H_r( z_p(G,X,*)).
$$
\end{definition}

In the case of the trivial group, we recover the set $X_{(p)}(r)$ of dimension $(p+r)$ subvarieties of $X \times \Delta^{r}$ meeting all faces properly, the cycle complex $z_{p}(X,*)$ and the higher Chow group $CH_{p}(X,r)$ defined in \cite{Bloch1} (or \cite{ML2}). 
Note that $CH_p(X,0)$ is exactly the usual Chow group $CH_{p}(X)$ of dimension $p$ cycles on $X$.

When $G$ acts $trivially$ on $X$, there is a natural isomorphism 
$$
K_0(X) \otimes K_0(G, \mathrm{Spec}(k)) \cong K_0(G, X)
$$ (cf. \cite[Proposition 1.6]{AV1}) which yields an isomorphism
$$
z_p(X, r) \otimes RG \cong z_p(G, X, r) \otimes \Lambda.
$$
Therefore,
\begin{equation} \label{TrivialAction}
CH_p(X, r) \otimes RG \cong CH_p(G, X, r) \otimes \Lambda.
\end{equation}

Since $X$ is quasi-projective and $G$ is finite, the quotient $X/G$ exists as a scheme. If $G$ acts $freely$ on $X$, then 
$
X^G_{(p)}(r) = (X/G)_{(p)}(r)
$ and $K_0(G_x, \mathrm{Spec} \ k(x)) =  K_0(\mathrm{Spec} \ k(x)) \cong \mathbb{Z}$. This gives a natural isomorphism
$$
z_p(G,X,r) \cong z_p(X/G,r),
$$
and hence
$$
CH_p(G,X,r) \cong CH_p(X/G,r).
$$

\subsection{Functoriality with respect to X}

If $f: Y \to X$ is a $proper \ G$-equivariant morphism in $\mathbf{Sch}^G_k$, there is the \textit{push-forward homomorphism} 
\begin{equation*}
f_{*}(r): z_p(G,Y,r) \to z_p(G,X,r)
\end{equation*} defined by
$$
[\alpha \in K_0 (G_z, \mathrm{Spec}(k(Z)))] \mapsto [( f \times \mathrm{id})_{*} (\alpha) \in K_0 (G_{( f \times \mathrm{id})(z)} , \mathrm{Spec}(k(( f \times \mathrm{id})(Z ))))]
$$
if $Z \mapsto (f \times \mathrm{id})(Z)$ is generically finite and sending $\alpha$ to zero if not. These maps form a simplical map
$$
f_{*}(-): z_p(G,Y,-) \to z_p(G,X,-)
$$
which yields the \textit{push-forward map} for equivariant higher Chow groups
\begin{equation} \label{PushForward}
f_{*}: CH_p(G,Y,r) \to CH_{p}(G,X,r).
\end{equation}

If $f:Y \to X$ is a $flat \ G$-equivariant morphism of relative dimension $d$, the \textit{pull-back homomorphism}
\begin{equation*}
f^{*}(r): z_p(G,X,r) \to z_{p+d}(G,Y,r)
\end{equation*}
is given by
$$
[\alpha \in K_0 (G_z, \mathrm{Spec}(k(Z)))] \to [( f \times \mathrm{id})^{*} (\alpha) \in 
\bigoplus_{[y] \in Y^G_{(p+d)}(r), (f\times \mathrm{id})(y) \in G.x} K_0 (G_y , \mathrm{Spec}(k(y))].
$$
These maps form a simplical map 
$$
f^{*}(-): z_p(G,X,-) \to z_{p+d}(G,Y,-)
$$
which yields the \textit{pull-back map} for equivariant higher Chow groups
\begin{equation} \label{PullBack}
f^{*}: CH_p(G,X,r) \to CH_{p}(G,Y,r).
\end{equation}

For proper composable morphisms $f$ and $g$, we have $(f \circ g)_* = f_* \circ g_*$. If $f$ and $g$ are flat composable morphisms then $(f \circ g)^* = g^* \circ f^*$. Moreover, if
\[
\begin{tikzcd}
W \arrow{r}{f'} \arrow{d}{g'}
&Z \arrow{d}{g}\\
Y \arrow{r}{f} &X
\end{tikzcd}
\]
is a $G$-equivariant cartesian squares in which $f$ is proper and $g$ is flat then $g^* \circ f_* = f'^* \circ g'_*$. These properties are well-known for (equivariant) algebraic $K$-theory and higher Chow groups.
See \cite[3.2]{LS1} for more details.

\subsection{Functoriality with respect to G}

If $\phi: H\to G$ is a group homomorphism and $X \in \mathbf{Sch}^G_k$ is a $G$-scheme then $\phi$ gives an action of $H$ on $X$. The collection of maps
$$
\phi_x^{*}: K_0(G_x, \mathrm{Spec} k(x)) \to K_0(H_x, \mathrm{Spec})
$$
induces the map of simplicial abelian groups  
$$
\phi^{*}: z_p(G, X, r) \to z_p(H,X,r)
$$
which yields the \textit{restriction map}
$$
\phi^{*}: CH_p(G, X, r) \to CH_p(H,X,r).
$$
It is easy to see that the maps $\phi^*$ are natural with respect to the proper push-forward and flat pull-back (cf. \cite[Section 3.2]{LS1}). 

We now want to construct a group homomorphism going in the other direction when $H \subset G$ be a subgroup. 

Recall that the \textit{induction morphism} for equivariant algebraic $K$-theory
$$
\mathrm{Ind}^G_H: K_0(H,X) \to K_0(G,X)
$$
is the push-forward
$$
\pi_{*}: K_0(G, (X\times G) / H) \to K_0(G,X)
$$
along the ($G$-equivariant) projection $\pi: (X \times G) / H \to X$ together with a natural identification 
$$
K_0(G, (X \times G) / H) \cong K_0(H,X)
$$ where the (free) action of $H$ on $X \times G$ is given by
$$
h(x,g) = (hx, gh^{-1}).
$$
For more details, see \cite[Section 2]{AV1}.

The analogous operations are valid for equivariant cycle complex. If $[x] \in (X \times \Delta^{r})_{(p+r)}/H$ such that $\overline{H.x} \in S^{H,X}_{(p)}(r)$ then $\overline{G.x} \in S^{G,X}_{(p)}(r)$. 
It is obvious that each orbit $H.x$ gives rise to an orbit $G.x$ and the stabilizer group $H_x$ is a subgroup of $G_x$. Therefore, we have the induction map for equivariant $K$-theory:
$$
\mathrm{Ind}^{G_x}_{H_x}(\mathrm{Spec}(k(x))): K_0(H_x, \mathrm{Spec} (k(x))) \to K_0(G_x, \mathrm{Spec} (k(x))).
$$
Take the sum over all orbits on both sides to obtain a map on the elements of the equivariant cycle complexes
$$
\mathrm{Ind}^G_H(X,r) : z_p(H,X,r) \to z_p(G,X,r).
$$ 
The maps $\mathrm{Ind}^G_H(X,r)$ are 'push-forward' maps, so they form a simplical map 
$$
\mathrm{Ind}^G_H(X,-) : z_p(H,X,-) \to z_p(G,X,-)
$$
which yields the \textit{induction map} for equivariant higher Chow groups
\begin{equation} \label{InducedMap}
\mathrm{Ind}^G_H(X) : CH_p (H, X, r) \to CH_p (G, X, r).
\end{equation}

\begin{lemma} \label{Commutativity}
$\mathrm{(1)}$ If $f: Y \to X$ is a flat $G$-morphism of $G$-schemes, then for any subgroup $H \subset G$ the diagram
\[
\begin{tikzcd}
CH_{p}(H,X,r) \arrow{r}{\mathrm{Ind}^G_H(X)} \arrow{d}{f^{*}}
&CH_{p}(G,X,r) \arrow{d}{f^{*}}\\
CH_{p}(H,Y,r) \arrow{r}{\mathrm{Ind}^G_H(Y)} & CH_{p}(G,Y,r)
\end{tikzcd}
\]
commutes.

$\mathrm{(2)}$ If $g: Y \to X$ is a proper $G$-morphism of $G$-schemes, then for any subgroup $H \subset G$ the diagram
\[
\begin{tikzcd}
CH_{p}(H,Y,r) \arrow{r}{\mathrm{Ind}^G_H(Y)} \arrow{d}{g_{*}}
&CH_{p}(G,Y,r) \arrow{d}{g_{*}}\\
CH_{p}(H,X,r) \arrow{r}{\mathrm{Ind}^G_H(X)} & CH_{p}(G,X,r)
\end{tikzcd}
\]
commutes.
\end{lemma}
\begin{proof}
The proof is straightforward, using the fact that the induction morphism
$$
\mathrm{Ind}^G_H(X) : CH_p (H, X, r) \to CH_p (G, X, r).
$$
is the push-forward
$$
\pi_{*}: CH_p (G, (X \times G)/H, r) \to CH_p (G, X, r).
$$
along the ($G$-equivariant) projection $\pi: (X \times G) / H \to X$, together with an identification 
$$
CH_p(G, (X \times G) / H, r) \cong CH_p(H,X, r).
$$
\end{proof}

\subsection{Equivariant motivic spectral sequence}
Equivariant higher Chow groups enjoy certain good properties as their ordinary counterparts.
\begin{proposition}[Localization theorem] \cite[Theorem 4.1]{LS1} \label{Localization}
Let $X \in \mathbf{Sch}^G_k$, $i: W \to X$ a $G$-invariant closed subscheme with open complement $j: U \to X$. Then for each $p$, there is a long exact sequence
$$
\ldots \to CH_p(G,W,r) \xrightarrow{i^*} CH_p(G,X,r) \xrightarrow{j^*} CH_p(G,U,r) \xrightarrow{\delta} CH_p(G,W,r-1) \rightarrow \ldots
$$
\end{proposition}
Using this, Levine and Serp\'e show that

\begin{proposition}[Equivariant motivic spectral sequence] \cite[Corollary 3.8]{LS1} \label{SpectralSequence}
Let $X \in \mathbf{Sch}^G_k$. There is a strongly convergent spectral sequence 
\begin{equation} \label{EquivariantSpectralSequence}
E^{p,q}_1 = CH_{-p}(G,X,-p-q) \Rightarrow G_{-p-q}(G,X).
\end{equation}
\end{proposition}

\section{Reconstruction Theorem}

From now on, we will assume that $k$ contains all the $n$-th roots of unity where $n:=|G|$. With our hypotheses, the ring $RG$ only depends on the characteristic of $k$. 
In this section, we will investigate the relation between equivariant higher Chow groups and the ordinary higher Chow groups of fixed point subschemes under cylic subgroups.

\subsection{Comparision morphism} For any cyclic subgroup $\sigma \subset G$ of order $m$,
let $t$ be the generator of the dual group $\hat{\sigma}$ of homomorphisms $\sigma \to k^*$.
We have
$$
R \sigma \cong \Lambda \hat{\sigma} \cong \Lambda[t]/(t^m-1) \cong \prod_{d|m} \Lambda[t]/(\Phi_d(t))
$$
where $\Phi_d(t)$ is the $d$-th \textit{cyclotomic polynomial}. Denote by $\tilde{R} \sigma$ the factor of $R \sigma$ corresponding to $\Lambda[t]/(\Phi_m(t))$ which is independent of the choice of the generator $t$.
The \textit{normalizer} $N(\sigma)$ of $\sigma$ acts naturally on $R\sigma$ and $\tilde{R}\sigma$. 
Moreover, $N(\sigma)$ acts on the \textit{fixed point subscheme} $X^{\sigma}$ of $X$ under $\sigma$ hence acts on $CH_{*}(X^{\sigma}, *)$.
Note that by our assumption, the functor $-\otimes \tilde{R}\sigma$ is exact on the cateogry of $\Lambda$-modules and 
for any subgroup $H \subset G$, taking $H$-invariants $(-)^H$ is an exact functor from the category of $\Lambda H$-modules to the category of $\Lambda$-modules.

Let $\Gamma$ be a set of representatives for the conjugacy classes of cyclic subgroups of $G$. We define the morphism 
\begin{equation} \label{Morphism}
\pi(X) : \prod_{\sigma \in \Gamma} (CH_p (X^{\sigma}, r) \otimes \tilde{R}\sigma)^{N(\sigma)} \to CH_p(G, X, r) \otimes \Lambda
\end{equation}
as the composition of the inclusion 
$$
\prod_{\sigma \in \Gamma} (CH_p (X^{\sigma}, r) \otimes \tilde{R}\sigma)^{N(\sigma)} \hookrightarrow \prod_{\sigma \in \Gamma} (CH_p (X^{\sigma}, r) \otimes R\sigma)^{N(\sigma)}
$$
given by the embedding $\tilde{R} \sigma  \hookrightarrow R \sigma$,
the isomorphism
$$
\prod_{\sigma \in \Gamma} (CH_p (X^{\sigma}, r) \otimes R\sigma)^{N(\sigma)} \xrightarrow{\sim} \prod_{\sigma \in \Gamma} ((CH_p (\sigma, X^{\sigma}, r))\otimes \Lambda)^{N(\sigma)}
$$
given by \eqref{TrivialAction},
the obvious inclusion 
$$
\prod_{\sigma \in \Gamma} (CH_p (\sigma, X^{\sigma}, r)\otimes \Lambda )^{N(\sigma)} \hookrightarrow \prod_{\sigma \in \Gamma} CH_p (\sigma, X^{\sigma}, r) \otimes \Lambda,
$$
the push-forward
$$
\prod_{\sigma \in \Gamma} CH_p (\sigma, X^{\sigma}, r)\otimes \Lambda \to \prod_{\sigma \in \Gamma} CH_p (\sigma, X, r) \otimes \Lambda
$$
along closed embeddings $X^{\sigma} \hookrightarrow X$,
and the product of induction maps \eqref{InducedMap}
$$
\prod_{\sigma \in \Gamma} CH_p (\sigma, X, r)\otimes \Lambda \to CH_p (G, X, r)\otimes \Lambda.
$$

Replace equivariant higher Chow groups by the $K$-groups of equivariant cohenrent sheaves everywhere, we obtain the homomorphism
\begin{equation} \label{FourthIdentity}
\alpha_{*}(X): \prod_{\sigma \in \Gamma} (G_{*}(X^{\sigma}) \otimes \tilde{R} \sigma)^{N(\sigma)} \to G_{*}(G,X) \otimes \Lambda
\end{equation}
defined by Vistoli \cite[Section 3]{AV1}. We will show that $\pi(X)$ is an isomorphism of $\Lambda$-modules.
\subsection{A simple case}

\begin{proposition} \label{FirstIdentity}
Let $X: = \coprod_{y} \mathrm{Spec}(k(y))$ be a $0$-dimensional scheme where $G$ acts transitively on a set of points then 
\begin{equation} \label{Simple}
\alpha_{0}(X): \prod_{\sigma \in \Gamma} (K_0(X^{\sigma}) \otimes \tilde{R}\sigma)^{N(\sigma)} \to K_0(G, X) \otimes \Lambda
\end{equation}
is an isomorphim.
\end{proposition}
\begin{proof}
If $G$ acts trivially on $X$, then $X = \mathrm{Spec}(F)$ for a field extension $F/k$. In this case, both sides of \eqref{Simple} are free $\Lambda$-modules whose rank are the number of conjugacy classes of $G$. 
The injectivity of $\alpha_{0}(\mathrm{Spec}(F))$ is proved in \cite[Proposition 1.5]{AV1}. 

In general, we fix an arbitrary element $x$. Since $G$ acts transitively on the index set, we have
$$
K_0(G,X) = K_0(G_x, \mathrm{Spec} \ k(x)).
$$
Let $\mathcal{C}$ be the set of cyclic subgroups of $G$ and $\mathcal{D}(y)$ the set of cyclic subgroups of $G_y$. We also have
$$
\prod_{\sigma \in \mathcal{C}} K_0(X^{\sigma}) \cong \prod_{y} \prod_{\sigma \in \mathcal{D}(y)} K_0(\mathrm{Spec} \ k(y))^{\sigma}).
$$

Denote $\Psi$ for a set of representatives for the conjugacy classes of cyclic subgroups of $G_x$ and let $\mathcal{D}: = \mathcal{D}(x)$.
By Lemma \ref{ElementaryLemma} below, there are isomorphisms
\begin{align*}
\prod_{\sigma \in \Gamma} (K_0(X^{\sigma}) \otimes \tilde{R}\sigma)^{N_G(\sigma)} & \cong (\prod_{\sigma \in \mathcal{C}} K_0(X^{\sigma}) \otimes \tilde{R}\sigma)^{G} \\
& \cong (\prod_{y} \prod_{\sigma \in \mathcal{D}} K_0(\mathrm{Spec}\ k(x))^{\sigma}) \otimes \tilde{R}\sigma)^{G} \\
& \cong (\prod_{\sigma \in \mathcal{D}} K_0((\mathrm{Spec} \ k(x))^{\sigma}) \otimes \tilde{R}\sigma)^{G_x} \\
& \cong \prod_{\sigma \in \Psi} (K_0((\mathrm{Spec} \ k(x))^{\sigma}) \otimes \tilde{R}\sigma)^{N_{G_x}(\sigma)}.
\end{align*}

Remark that if $\sigma$ acts non-trivially on $\mathrm{Spec} \ k(x)$ then $(\mathrm{Spec} \ k(x))^{\sigma} = \emptyset$. 
Hence, by replacing $G$ with $G_x$, $X$ with $\mathrm{Spec} \ k(x)$ and changing notation, one reduces the statement to the case of a point, i.e., to prove that 
\begin{equation} \label{SecondIdentity}
\prod_{\sigma \in \Gamma'} (K_0(\mathrm{Spec} \ F) \otimes \tilde{R}\sigma)^{N_G(\sigma)} \cong K_0(G, \mathrm{Spec} \ F) \otimes \Lambda 
\end{equation}
where $F/k$ is a finitely generated field extension and $\Gamma'\subset \Gamma$ is the subset consisting of cyclic subgroups of $G$ which act $trivially$ on $F$. 

Denote $G_{(F)}$ for the \textit{inertia group} of $\mathrm{Spec} \ F$, i.e.,
$$
G_{(F)}: = \mathrm{ker}(G \to \mathrm{Aut}(F / k)),
$$
and $\mathcal{I}$ for the set of cyclic subgroup of $G_{(F)}$. We have
\begin{align*}
\prod_{\sigma \in \Gamma'} (K_0(\mathrm{Spec} \ F) \otimes \tilde{R}\sigma)^{N_G(\sigma)} \cong & (\prod_{\sigma \in \mathcal{I}} K_0(\mathrm{Spec} \ F) \otimes \tilde{R}\sigma)^G \\
\cong & ((\prod_{\sigma \in \mathcal{I}} K_0(\mathrm{Spec} \ F) \otimes \tilde{R}\sigma)^{G_{(F)}})^{G/G_{(F)}} \\
\cong & (K_0(G_{(F)}, \mathrm{Spec} \ F)\otimes \Lambda)^{G/G_{(F)}} \\
\cong & K_0(G,\mathrm{Spec} \ F) \otimes \Lambda.
\end{align*}
The first isomorphism holds by Lemma \ref{ElementaryLemma}. The third isomorphism holds because $G_{(F)}$ acts trivially on $F$. The last isomorphism is a special case of the descent property for equivariant algebraic $K$-theory with $\Lambda = \mathbb{Z}[1/n]$-coefficients. Hence \eqref{SecondIdentity} is an isomorphism as claimed.
\end{proof}

\begin{lemma} \label{ElementaryLemma}
Let the group $G$ act on the left on a set $\mathcal{S}$ and let $\mathcal{T}$ be a set of representatives for the orbits. Assume that $G$ acts on the left on a product of abelian
groups of the type $\prod_{s \in \mathcal{S}} M_s$ in such a way that for any $g \in G$
$$
gM_s = M_{gs}.
$$
For each $t\in \mathcal{T}$ let $G_t$ be the stabilizer of $t$ in $G$. Then the canonical projection
$$
\prod_{s \in \mathcal{S}} M_s \to \prod_{t \in \mathcal{T}} M_t
$$
induces an isomorphism
$$
(\prod_{s \in \mathcal{S}} M_s)^G \to \prod_{t \in \mathcal{T}} (M_t)^{G_t}.
$$
\end{lemma}
\begin{proof}
It is straightforward.
\end{proof}
In particular, if $\mathcal{C}$ is the set of cyclic subgroups of $G$ and $\Gamma$ is a set of representatives for the conjugacy classes of cyclic subgroups of $G$, then 
$$
\prod_{\sigma \in \Gamma} (K_0(X^{\sigma}) \otimes \tilde{R}\sigma)^{N_G(\sigma)} \cong (\prod_{\sigma \in \mathcal{C}} K_0(X^{\sigma}) \otimes \tilde{R}\sigma)^{G} \cong 
(\prod_{g \in G} K_0(X^{g}) \otimes \Lambda)^{G}.
$$

\subsection{The general case}

\begin{theorem}[Reconstruction Theorem] \label{MainTheorem}
The map $\pi(X)$ in \eqref{Morphism} is an isomorphism for any $X \in \mathbf{Sch}^G_k$.
\end{theorem}
\begin{proof}
Both sides of \eqref{Morphism} satisfy localization and $\pi(-)$ is natural with localization. 
Indeed, by Lemma \ref{Commutativity}, each component of $\pi$ is compatible with the push-forward $i_{*}$ given by closed immersion $i: Z \hookrightarrow X$ and the pull-back $j^{*}$ given by open immersion $j: U \hookrightarrow X$. 
Hence, we only need to show that $\pi(X)$ is an isomorphism for $X = \coprod_x \mathrm{Spec}(k(x))$ when $G$ acts transitively on a set of points. 

By the same argument as in Proposition \ref{FirstIdentity}, we reduce our problem to the case of a point, i.e., $\pi(\mathrm{Spec}(F))$ is an isomorphism for any finitely generated field extension $F/k$. 
In this case
$$
(CH_p( (\mathrm{Spec} \ F)^{\sigma}, r) \otimes \tilde{R}\sigma)^{N(\sigma)} = (CH_p(\mathrm{Spec} \ F), r) \otimes \tilde{R}\sigma)^{N(\sigma)}
$$
if $\sigma$ acts trivially on $\mathrm{Spec}(F)$ and 
$$
(CH_p((\mathrm{Spec} \ F)^{\sigma}, r) \otimes \tilde{R}\sigma)^{N(\sigma)} = 0
$$
otherwise.

Denote by $\Gamma'$ the subset of $\Gamma$ consisting of cyclic subgroups of $G$ which act trivially on $F$.
We have
\begin{align*}
\prod_{\sigma \in \Gamma'} (z_p (\mathrm{Spec} \ F, r)\otimes \tilde{R}\sigma)^{N(\sigma)} 
& = \prod_{\sigma \in \Gamma'}((\bigoplus_{x \in (\mathrm{Spec} \ F)_{(p)}(r)} \mathbb{Z}.x)\otimes \tilde{R}\sigma)^{N(\sigma)} \\
& = \prod_{\sigma \in \Gamma'} ((\bigoplus_{x \in (\mathrm{Spec} \ F)_{(p)}(r)} K_0(\mathrm{Spec} \ k(x)))\otimes \tilde{R}\sigma)^{N(\sigma)} \\
& = \prod_{\sigma \in \Gamma'} (\bigoplus_{[x] \in (\mathrm{Spec} \ F)^G_{p}(r)} \bigoplus_{y \in G.x} K_0(\mathrm{Spec} \ k(y)) \otimes \tilde{R}\sigma)^{N(\sigma)} \\
& = \prod_{\sigma \in \Gamma'} (\bigoplus_{[x]} K_0(\coprod_{y \in G.x}\mathrm{Spec} \ k(y)) \otimes \tilde{R}\sigma)^{N(\sigma)} \\
& = \bigoplus_{[x]} \prod_{\sigma \in \Gamma'} (K_0(\coprod_{y \in G.x}\mathrm{Spec} \ k(y))\otimes \tilde{R}\sigma)^{N(\sigma)}.
\end{align*}
If $\sigma \notin \Gamma'$, i.e., $\sigma$ acts non-trivially on $F$, then the scheme $\coprod_{y \in G.x} \mathrm{Spec} \ k(y)$ with $[x] \in (\mathrm{Spec} \ F)^G_{p}(r)$ has no fixed point subscheme under $\sigma$. 
If $\sigma \in \Gamma'$ then $\sigma$ acts trivially on $\coprod_{y \in G.x} \mathrm{Spec} \ k(y)$. Hence,
\begin{align*}
\bigoplus_{[x]} \prod_{\sigma \in \Gamma'} (K_0(\coprod_{y \in G.x}\mathrm{Spec} \ k(y))\otimes \tilde{R}\sigma)^{N(\sigma)}) 
& = \bigoplus_{[x]} \prod_{\sigma \in \Gamma} (K_0((\coprod_{y \in G.x}\mathrm{Spec} \ k(y))^{\sigma})\otimes \tilde{R}\sigma)^{N(\sigma)}) \\
& = \bigoplus_{[x]} K_0(G, \coprod_{y \in G.x}\mathrm{Spec}  \ k(y)) \otimes \Lambda \ ({\rm by \ Proposition}\ \ref{FirstIdentity})\\
& = \bigoplus_{[x]} K_0(G_x, \mathrm{Spec} \ k(x)) \otimes \Lambda \\
& = z_p(G,\mathrm{Spec}\ F, r) \otimes \Lambda.
\end{align*}
Take homologies of the associated complexes, one obtains isomorphisms
$$
\prod_{\sigma \in \Gamma}(CH_p(\mathrm{Spec}\ F, r)\otimes \tilde{R}\sigma)^{N(\sigma)} \cong CH_p(G, \mathrm{Spec}\ F, r) \otimes \Lambda.
$$
for each $p$ and $r$.
\end{proof}

\section{Applications}
Levine has constructed in \cite{ML2} a spectral sequence
\begin{equation} \label{MotivicSpectraSequence}
E^{p,q}_1 = CH_{-p}(X,-p-q) \Rightarrow G_{-p-q}(X)
\end{equation}
for any quasi-projective scheme $X$ over $k$. In this section we will show that the isomorphism \eqref{Morphism} is compatible with the spectral sequences \eqref{MotivicSpectraSequence} and\eqref{SpectralSequence}. 
We then show that the equivariant motivic spectral sequence \eqref{SpectralSequence} degenerates rationally to obtain a Riemann-Roch theorem for equivariant algebraic $K$-theory.

\subsection{Comparision morphism revisited}
In Section 3 we defined the homomorphism \eqref{FourthIdentity}
\begin{equation*}
\alpha_{*}(X): \prod_{\sigma \in \Gamma} (G_{*}(X^{\sigma}) \otimes \tilde{R} \sigma)^{N(\sigma)} \to G_{*}(G,X) \otimes \Lambda.
\end{equation*}
We show now that this morphism can actually be expressed on the level of spectra.

Assume that $H$ is a finite group which acts \textit{trivially} on a $k$-scheme $Y$ (in our application, $\sigma$ acts trivially on $X^{\sigma}$).
Let $\mathcal{I}$ be the set of \textit{irreducible representations} of $H$ over $k$. For any representation $\rho$ in $\mathcal{I}$ with representation space $V$, we define
\begin{equation*}
\rho^* : \mathbf{Coh}(Y) \to \mathbf{Coh}_H(Y)
\end{equation*}
by mapping any coherent sheaf $\mathcal{F}$ on $Y$ to $\mathcal{F} \otimes_k V$ with the action of $H$ induced by the action of $H$ on $V$. It is obvious that $\rho^*$ is an exact functor, hence it induces a map between spectra
\begin{equation*}
\rho^*: G(Y) \to G(H,Y).
\end{equation*}

Assume further that there is another group $\Pi$ acting (possibly \textit{non-trivially}) on $Y$ and $H$ in such a way that the action respects the group structure of $H$, i.e.,
$$
\alpha(h_1.h_2) = \alpha h_1 . \alpha h_2
$$
for any $\alpha \in \Pi$ and $h_1, h_2 \in H$. 
This induces actions of $\Pi$ on the spectra $G(Y)$, $G(H,Y)$ and on the set $\mathcal{I}$.
Let $\mathcal{J}$ be a subset of $\mathcal{I}$ which is closed under the action of $\Pi$, then $\prod_{\rho \in \mathcal{J}} \mathbf{Coh}(Y)$ and $\mathbf{Coh}_H(Y)$ have a natural action of $\Pi$ and the map
\begin{equation*}
\prod_{\rho \in \mathcal{J}} \rho^*: \prod_{\rho \in \mathcal{J}} \mathbf{Coh}(X) \to \mathbf{Coh}_H(X)
\end{equation*}
is an $\Pi$-equivariant map. This induces an $\Pi$-equivariant map between spectra
\begin{equation} \label{ThirdIdentity}
\prod_{\rho \in \mathcal{J}} \rho^{*}: \prod_{\mathcal{J}} G(Y) \to G(H,Y).
\end{equation}
Moreover, the inclusion $\mathcal {J} \hookrightarrow \mathcal{I}$ induces an $\Pi$-equivariant map
$$
\prod_{\mathcal{J}}G(Y) \to \prod_{\mathcal{I}}G(Y)
$$
and the following diagram
\[
\begin{tikzcd}
\prod_{\mathcal{J}}G(Y) \arrow[r, hook] \arrow[dr]
& \prod_{\mathcal{I}}G(Y) \arrow[d]\\
& G(H,Y)
\end{tikzcd}
\]
commutes. Hence we obtain a commutative diagram between \textit{homotopy fixed point spectra}
\[
\begin{tikzcd}
(\prod_{\mathcal{J}}G(Y))^{h\Pi} \arrow[r, hook] \arrow[dr]
& (\prod_{\mathcal{I}}G(Y))^{h\Pi} \arrow[d]\\
& (G(H,Y))^{h\Pi}
\end{tikzcd}
\]

Replacing $H$ by $\sigma$, $Y$ by $X^{\sigma}$ and $\Pi$ by $N(\sigma)$. Since $k$ contains enough roots of unity,
it is well-known that the set of irreducible representations of $\sigma$ over $k$ has $m$ elements where $m$ is the order of $\sigma$ and every irreducible representation of $\sigma$ is one dimensional given by multiplication by a $m$-th root of unity.
Let $\mathcal{J}$ be the set of representations given by multiplication by a $primitive \ m$-th root of unity. It is clear that $|\mathcal{J}| = \mathrm{deg} (\Phi_m(t))$, the degree of the $m$-th cyclotomic ring $\Phi_m(t)$.
Define the map 
\begin{equation} \label{5thIdentity}
\alpha(X): \prod_{\sigma \in \Gamma} (\prod_{\mathcal{J}} G(X^{\sigma})[1/n])^{hN(\sigma)} \to G(G,X)[1/n]
\end{equation}
($n = |G|$ is inverted) as the composition of the inclusion
$$
\prod_{\sigma} (\prod_{\mathcal{J}} G(X^{\sigma})[1/n])^{hN(\sigma)} \to \prod_{\sigma} (\prod_{\mathcal{I}} G(X^{\sigma})[1/n])^{hN(\sigma)}
$$
given by the inclusion $\mathcal{J} \hookrightarrow \mathcal{I}$, the morphism
$$
\prod_{\sigma} (\prod_{\mathcal{I}} G(X^{\sigma})[1/n])^{hN(\sigma)} \to \prod_{\sigma} (G(\sigma, X^{\sigma})[1/n])^{hN(\sigma)}
$$
given by \eqref{ThirdIdentity}, the natural morphism
$$
\prod_{\sigma} (G(\sigma, X^{\sigma})[1/n])^{hN(\sigma)} \to \prod_{\sigma} G(\sigma, X^{\sigma})[1/n],
$$
the push-forward
$$
\prod_{\sigma} (G(\sigma, X^{\sigma}))[1/n] \to \prod_{\sigma} G(\sigma, X)[1/n]
$$
given by the closed embedding $X^{\sigma} \hookrightarrow X$, 
and the wedge sum of induction maps
$$
\prod_{\sigma} G(\sigma, X)[1/n] \to G(G,X)[1/n].
$$

\begin{proposition}
The map $\alpha(X)$ of \eqref{5thIdentity} induces $\alpha_{*}(X)$ of \eqref{FourthIdentity} on homotopy groups.
\end{proposition}
\begin{proof}
We only have to show that 
$$
\pi_d (\prod_{\mathcal{J}} G(X^{\sigma})[1/n])^{hN(\sigma)} \cong (\prod_{\mathcal{J}} G_d(X^{\sigma}) \otimes \tilde{R} \sigma)^{N(\sigma)} 
$$
and
$$
\pi_d (G(G,X)[1/n]) \cong G_d(G,X) \otimes \Lambda.
$$
The last identity is clear by definition. For the first identity, note that if $S$ is a spectrum with an action of a finite group $G$ of order $n$, then there is a natural isomorphism
\begin{equation} \label{6thIdentity}
\pi_d (S^{hG}[1/n]) \cong (\pi_d S)^{G} \otimes \mathbb{Z}[1/n].
\end{equation}
Indeed, the spectral sequence
\begin{equation*}
E_2^{p,q} = H^{-p} (G, \pi_q(S) \otimes \mathbb{Z}[1/n]) \Rightarrow \pi_{p+q} (S^{hG}[1/n])
\end{equation*}
has $E_2^{p,q} = (\pi_q(S))^G \otimes \mathbb{Z}[1/n]$ if $p=0$ and equals $0$ otherwise.
Moreover, it is easy to see that 
\begin{equation} \label{7thIdentity}
\pi_d (\prod_{\mathcal{J}} G(X^{\sigma})[1/n]) = \prod_{\mathcal{J}} G_d(X^{\sigma}) \otimes \mathbb{Z}[1/n] \cong G_d(X^{\sigma}) \otimes \tilde{R}\sigma
\end{equation}
with a compatible action of $N(\sigma)$ on both sides. Our claim follows from \eqref{6thIdentity} and \eqref{7thIdentity} .
\end{proof}

\subsection{A Riemann-Roch theorem}
The way we define $\alpha(X)$ extends naturally to a family of compatible maps $\alpha_{(*)}(X)$ between the two towers
\[
\begin{tikzcd} 
\prod_{\sigma}(\prod_{\mathcal{J}} G_{(p)}(X^{\sigma},-)[1/n])^{hN(\sigma)} \arrow{d}{\alpha_{(p)}(X)} \arrow{r} & \ldots  \arrow{r} & \prod_{\sigma}(\prod_{\mathcal{J}}(X^{\sigma})[1/n])^{hN(\sigma)} \arrow{d}{\alpha(X)}\\
G_{(p)}(G,X,-)[1/n] \arrow{r} & \ldots  \arrow{r} & G(G,X)[1/n].
\end{tikzcd}
\]
The notation $G_{(p)}(G,X,-)$ used here is inherited from \cite[Section 2.1]{LS1}. Namely, $G_{(p)}(G,X,-)$ is the simplicial spectrum whose $r$-simplices is
$$
G_{(p)}(G,X,r): = \mathrm{hocolim}_{W} \ G^{W}(G,X\times \Delta^r)
$$
the $K$-theory of equivariant coherent sheaves on $X \times \Delta^r$ with supports in $W$ 
where $W$ runs over all the closed $G$-stable subsets of $X \times \Delta^{r}$ such that $\mathrm{dim}_k \ W \cap X \times F \le p + \mathrm{dim}_k \ F$ for all faces $F \subset \Delta^r$.

The top tower induces the spectral sequence
\begin{equation} \label{1SpectralSequence}
E^{p,q}_1 = \prod_{\sigma} (CH_{-p}(X^{\sigma},-p-q) \otimes \tilde{R} \sigma)^{N(\sigma)} \Rightarrow \prod_{\sigma} (G_{-p-q}(X^{\sigma}) \otimes \tilde{R}\sigma)^{N(\sigma)}
\end{equation}
by applying operations to \eqref{MotivicSpectraSequence}.

The bottom tower induces the spectral sequence
\begin{equation} \label{2SpectralSequence}
E^{p,q}_1 = CH_{-p}(G,X,-p-q) \otimes \Lambda \Rightarrow G_{-p-q}(G,X) \otimes \Lambda
\end{equation}
which is essentially the spectral sequence \eqref{SpectralSequence} with $\Lambda = \mathbb{Z}[1/n]$-coefficients.

This shows that the map $\pi(X)$ (induced by the family $(\alpha_{(p)}(X))_p$) is natural with respect to the spectral sequences \eqref{1SpectralSequence} and \eqref{2SpectralSequence}. 
Hence, one obtains Vistoli’s reconstruction theorem for equivariant algebraic $G$-theory:
\begin{corollary} {\rm \cite[Theorem 2]{AV1}} \label{Vistoli}
For any $X \in \mathbf{Sch}^G_k$ the map
$$
\alpha_*(X): \prod_{\sigma \in \Gamma} (G_{*}(X^{\sigma}) \otimes \tilde{R} \sigma)^{N(\sigma)} \to G_{*}(G,X) \otimes \Lambda
$$
is an isomorphism of graded $\Lambda$-modules which is compatible with localization sequence.
\end{corollary}

The following lemma is a weak version of \cite[Corollary 5.6]{LS1} but the proof is simpler:
\begin{corollary} 
Let $X \in \mathbf{Sch}^G_k$ and let $G$ act on $X \times \mathbb{A}^1$ via the given action on $X$ and the trivial action on $\mathbb{A}^1$. 
Then the pull-back via the projection $p: X \times \mathbb{A}^1 \to X$ induces an isomorphism
\begin{equation}
p_X^{*}: CH_p(G, X, r) \otimes \Lambda \to CH_{p+1}(G,X\times \mathbb{A}^1, r) \otimes \Lambda.
\end{equation}
\end{corollary}
\begin{proof}
Since $G$ acts trivially on $\mathbb{A}^1$, we have $(X \times \mathbb{A}^1)^{\sigma} = X^{\sigma} \times \mathbb{A}^1$. The diagram
\[
\begin{tikzcd}
\prod_{\sigma \in \Gamma} (CH_p (X^{\sigma}, r) \otimes \tilde{R}\sigma)^{N(\sigma)} \arrow{r}{\pi(X)} \arrow{d}{\prod p_{X^{\sigma}}^{*}}
& CH_p(G, X, r) \otimes \Lambda \arrow{d}{p_X^{*}}\\
\prod_{\sigma \in \Gamma} (CH_{p+1} (X^{\sigma}\times \mathbb{A}^1, r) \otimes \tilde{R}\sigma)^{N(\sigma)} \arrow{r}{\pi(X \times \mathbb{A}^1)} & CH_{p+1}(G, X \times \mathbb{A}^1, r) \otimes \Lambda
\end{tikzcd}
\]
is commutative. The maps $\pi(X)$ and $\pi(X \times \mathbb{A}^1)$ are isomorphisms by Theorem \ref{MainTheorem}. 
Each map $p_{X^{\sigma}}^{*}$ is an isomorphism by homotopy invariance for higher Chow groups \cite[Theorem (2.1)]{Bloch1}. 
Therefore, $p_X^{*}$ is an isomorphism.
\end{proof}

The reconstruction theorem for equivariant higher Chow groups yields a Riemann-Roch theorem for equivariant algebraic $K$-theory:
\begin{theorem} [Riemann-Roch]\label{BigConsequence}
If $X$ is a smooth $G$-scheme over $k$ then the spectral sequence \eqref{SpectralSequence} degenerates rationally, i.e., there are isomorphisms of $\mathbb{Q}$-vector spaces
\begin{equation} \label{Degeneration}
\bigoplus_{q}K_q(G,X) \otimes \mathbb{Q} \cong \bigoplus_{q}G_{q}(G,X) \otimes \mathbb{Q} \cong \bigoplus_{p,q} CH_{p}(G,X,q) \otimes \mathbb{Q}.
\end{equation}
\end{theorem}
\begin{proof}
Since $X$ is smooth and $G$ is finite, every $G$-equivariant coherent sheaf on $X$ has a resolution by $G$-equivariant locally free sheaves (\cite[Corollary 5.8]{RT1}). 
Quillen's \textit{Dévissage theorem} provides for any $q \in \mathbb{N}$ an isomorphism
$$
K_q(G,X) \cong G_q(G,X).
$$
For any $\sigma \in \Gamma$, the scheme $X^{\sigma}$ is smooth (cf. \cite[Proposition 3.4]{Edixhoven}). Hence, the spectral sequence 
\begin{equation*}
E^{p,q}_1 = CH_{-p}(X^{\sigma},-p-q) \Rightarrow G_{-p-q}(X^{\sigma})
\end{equation*} degenerates rationally to get
$$
\bigoplus_{p,q} CH_{p}(X^{\sigma},q) \otimes \mathbb{Q} \cong \bigoplus_{q}G_{q}(X^{\sigma}) \otimes \mathbb{Q}
$$
(cf. \cite[Theorem 14.8]{ML1}). Taking product over all $\sigma \in \Gamma$ yields
$$
\prod_{\sigma \in \Gamma} (\bigoplus_{p,q} (CH_{p}(X^{\sigma},q) \otimes \mathbb{Q}) \otimes_{\Lambda} \tilde{R} \sigma)^{N(\sigma)} \cong \prod_{\sigma \in \Gamma} (\bigoplus_{q} (G_{q}(X^{\sigma}) \otimes \mathbb{Q}) \otimes_{\Lambda} \tilde{R}\sigma)^{N(\sigma)}.
$$
Equivalently,
$$
\bigoplus_{p,q} (\prod_{\sigma \in \Gamma} CH_{p}(X^{\sigma},q)) \otimes \tilde{R} \sigma)^{N(\sigma)}\otimes_{\Lambda} \mathbb{Q} \cong \bigoplus_{q}( \prod_{\sigma \in \Gamma} G_{q}(X^{\sigma}) ) \otimes \tilde{R}\sigma)^{N(\sigma)}\otimes_{\Lambda} \mathbb{Q},
$$
i.e.,
$$
\bigoplus_{p,q} CH_{p}(G,X,q) \otimes_{\mathbb{Z}} \mathbb{Q} \cong \bigoplus_{q}G_{q}(G,X) \otimes_{\mathbb{Z}} \mathbb{Q}.
$$

\end{proof}
\begin{remark}
When $X \in \mathbf{Sch}^G_k$ which is not neccessarily smooth, there is still an isomorphism
\begin{equation*}
\bigoplus_{q}G_{q}(G,X) \otimes \mathbb{Q} \cong  \bigoplus_{p,q} CH_{p}(G,X,q) \otimes \mathbb{Q}
\end{equation*}
by Theorem \ref{MainTheorem}, Corollary \ref{Vistoli} and Bloch-Riemann-Roch isomorphism
$$
G_q(X^{\sigma}) \otimes \mathbb{Q} \xrightarrow{\sim} \bigoplus_{p} CH_{p}(X^{\sigma},q) \otimes \mathbb{Q}
$$
\cite[Theorem (9.1)]{Bloch1}. However, we do not know how to define directly an \textit{equivariant Riemann-Roch map}
$$
G_{q}(G,X) \rightarrow  \bigoplus_{p} CH_{p}(G,X,q) \otimes \mathbb{Q}.
$$
\end{remark}
\subsection{Multiplicative structure and further remarks} When $X$ is smooth, there are certainly two ways to equip a multiplicative structure on equivariant higher Chow groups with rational coefficients (the integral case is still unknown).

The first structure is obtained through the isomorphism \eqref{Degeneration} where the multiplication on $CH_{*}(G,X,*) \otimes \mathbb{Q}$ inherits the ordinary multiplication on $K_{*}(G,X)\otimes \mathbb{Q}$ (given by tensor product over $\mathcal{O}_X$). 

The second structure is obtained by the reconstruction theorem for equivariant higher Chow groups. Recall that Jarvis-Kaufmann-Kimura have considered in \cite{JKK1} the \textit{stringy Chow ring}
$$
\prod_{g \in G} CH_{*} (X^{g}) \otimes \mathbb{Q}
$$
with the multiplicative structure given by the \textit{stringy product} \cite[Definition 1.6]{JKK1}. This can be extended to define a multiplicative structure on 
$$
\prod_{g \in G} CH_{*}(X^{g},*) \otimes \mathbb{Q},
$$
and hence on the algebra of invariants
$$
(\prod_{g \in G} CH_{*}(X^{g},*) \otimes \mathbb{Q})^G.
$$
By Theorem  \ref{MainTheorem}, this gives another multiplicative structure on $CH_{*}(G,X,*)\otimes \mathbb{Q}$.

These two structures are different in general even in the case of Chow groups $CH_{*}(G,X,0)\otimes \mathbb{Q}$. 
The first one is the usual multiplication on the Grothendieck group of the quotient stack $[X/G]$. 
The second one should be the usual multiplication on the Grothendieck group of a (and all) hyper-Kähler resolution of the coarse moduli space $X/G$ of $[X/G]$. 
This is the content of the \textit{K-theoretic hyper-Kähler resolution conjectures} \cite[Conjecture 1.2]{JKK1} which has been verified in certain interesting cases.
\begin{remark}
As the reader might guess, the isomorphism $\pi(X)$ in Theorem \ref{MainTheorem} need not to be true in general if we index equivariant higher Chow groups by codimension rather than dimension as in \cite{LS1}. 
Some degree shifts are needed to get right indexes. The reason is that $\pi(X)$ is defined by using push-forward for algebraic cycles and $K$-theory that do not preserve codimension in general.  
For instance, let $G = \mathbb{Z}/2$ act on $X = \mathbb{A}^1: = \mathrm{Spec}(k[t])$ by sending $t \to -t$. A simple calculation gives 
$$
CH^1(\mathbb{Z}/2, \mathbb{A}^1, 0) = \mathbb{Z}.
$$
(cf. \cite[Example 6.17]{LS1}).
The group $\mathbb{Z}/2$ has only two (cyclic) subgroups $0$ and $\mathbb{Z}/2$. We have
$$
CH^1((\mathbb{A}^1)^0, 0) = CH^1(\mathbb{A}^1, 0) = CH^1(\mathrm{Spec}k, 0) = 0
$$
and 
$$
CH^1((\mathbb{A}^1)^{\mathbb{Z}/2}, 0) = CH^1(\mathrm{Spec}k, 0) = 0
$$
by dimension reason. Therefore, 
$$
\prod_{\sigma}(CH^1((\mathbb{A}^1)^{\sigma}, 0) \otimes \tilde{R}\sigma)^{N(\sigma)} = 0.
$$
\end{remark}

It is aslo worth mentioning that the two multiplicative structures on rational equivariant higher Chow groups considered above do not respect the grading by codimension. 
For the first multiplicative structure, the reason is that the ring structure on $K_0(G,X)$ does not in general respect the topological filtration \cite[Remark 3.6]{LS1}.
For the second multiplicative structure, this failure is measured by 'age' (or 'degree shifting number') \cite[Definition 1.3]{JKK1}.
\begin{remark}
In the non-equivariant case for $X$ smooth, the spectral sequence $\eqref{MotivicSpectraSequence}$ admits actions of \text{Adams operations} which implies its degeneration with 
rational coefficients. It would be interesting to see how to equip Adam operations on the equivariant motivic spectral sequence \eqref{EquivariantSpectralSequence}. 
It might be possible to follow the construction given in \cite[Theorem 9.7]{ML1}, but there are some technical annoyances we have to overcome. 
Even if we are in a good situation, there is no reason to expect that $\pi(X)$ preserves these operations because it is defined using push-forwards which are not ring homomorphisms in general.

\end{remark}
\bibliographystyle{amsplain}

\end{document}